\documentclass{amsart}
\usepackage{graphicx}
\usepackage{amssymb}

\DeclareMathOperator{\Ext}{Ext}
\newcommand{\Z}{\mathbb Z}

\newcommand{\Q}{\mathbb Q}
\newcommand{\F}{\mathbb F}

\newcommand{\G}{\mathbb G}

\newcommand{\toda}[1]{\langle #1\rangle}

\newtheorem{thm}{Theorem}[section]

\newtheorem{lem}[thm]{Lemma}
\newtheorem{cor}[thm]{Corollary}
\newtheorem{prop}[thm]{Proposition}

\newtheorem*{remark}{Remark}

\title{The $5$-local Homotopy of $eo_4$}
\author{Michael A.~Hill}

\begin{document}
\bibliographystyle{amsplain}

\begin{abstract}
We compute the $5$-local cohomology of a particular Hopf algebroid related to algebraic topology and algebraic geometry. This computation arises in a number of settings, and it is a $5$-local analogue of the Weierstrass Hopf algebroid used to compute $tmf$ homology. We finish the paper by computing Adams-Novikov differentials in the cohomology, giving the homotopy, $V(0)$-homology, and $V(1)$-homology of the putative spectrum $eo_4$.
\end{abstract}

\maketitle

\section{Introduction}\label{sec:Intro}
We present a computation of the cohomology of the Gorbounov-Hopkins-Mahowald Hopf algebroid at the prime $5$. This Hopf algebroid is given by
\[
(A,\Gamma)=(\Z_p[a_1,\dots,a_p], A[r]),
\]
where $|a_i|=2i(p-1)$ and $|r|=2(p-1)$. The left unit is given by the canonical inclusion, and the right unit is given by the formula
\[
\eta_R(a_i)=\sum_{j=0}^i\binom{p-j}{i-j} a_{j}r^{i-j}.
\]
The element $r$ is primitive.

The origin of the formulas for the right units and coproduct on $r$ is geometric. The Hopf algebroid $(A,\Gamma)$ corepresents curves of the form
\begin{equation}\label{eqn:DefEqn}
y^{p-1}=x^p+a_1x^{p-1}+a_2x^{p-2}+\dots+a_p,
\end{equation}
together with coordinate transformations $x\mapsto x+r$ (we recover the grading if we further consider a scaling action $(x,y)\mapsto(\lambda^{p-1}x,\lambda^p y)$). The coefficients $a_i$ can then be interpreted as functions on the moduli stack of all such curves. The elements of $A$ that are invariant under translation are the analogue of modular forms in this context, and one of the goals of this paper is to compute the full ring of invariants for this Hopf algebroid using the full cohomology ring.

There is a distinguished invariant class in dimension $2p(p-1)^2$. Associated to any curve of the form (\ref{eqn:DefEqn}) is the discriminant $\Delta$, a polynomial in the coefficients which is invariant under coordinate transformations. For $p=2$, this is the classical formula $a_1^2-4a_2$, and for $p=3$, this is the modular form of the same name.

The Gorbounov-Hopkins-Mahowald Hopf algebroid is a generalization to all primes of the Weierstrass Hopf algebroid used to compute the Adams-Novikov $E_2$-term for the $3$-local homotopy of $tmf$ \cite{HopkinsMahowaldEO2, Tilmantmf}. Building on this example, these Hopf algebroids were introduced to better make computations of the homotopy of the higher real $K$-theories $EO_{p-1}$ at $p$ \cite{GorMah}.

Hopkins and Miller defined spectra $EO_{p-1}$ as $E_{p-1}^{hG}$, where $E_{p-1}$ is the Lubin-Tate spectrum of height $(p-1)$ and where $G$ is a maximal finite subgroup of the Morava stabilizer group $\G_{p-1}$ \cite{HennFiniteKnResolutions}. When $p=2$, this gives the $2$-completion of $KO$, and when $p=3$, this gives the $K(2)$-localization of $TMF$, and in these two cases, there are connective models: $ko$ and $tmf$. These spectra $K(p-1)$-localize to $EO_{p-1}$, and their Adams-Novikov $E_2$ term is given by $\Ext$ over $(A,\Gamma)$. For primes bigger than $3$, there is still no construction of an analogous connective model $eo_{p-1}$. Just as for $p=2$ or $3$, we should have that $L_{K(p-1)}(eo_{p-1})=EO_{p-1}$, and the Adams-Novikov $E_2$ term for $eo_{p-1}$ should be given by $\Ext_{(A,\Gamma)}(A,A)$.

For the remainder of this paper, we will assume the existence of such a spectrum, assuming moreover that it is a ring spectrum and an f.p. spectrum in the sense of Mahowald-Rezk \cite{MahowaldRezkFP}. In \S~\ref{sec:EOp-1Intro}, we will ground these computations, showing that the algebraic computations determine completely $\pi_\ast EO_4$, even if an $eo_4$ of the desired form does not exist.

\subsection{Organization}
In \S\ref{sec:HopfAlg}, we provide the computation of the rational cohomology of the Hopf algebroid (giving, in particular, the homotopy of the rationalization of $eo_{p-1}$), and we state our main results. In \S\ref{sec:AllPrimes} through \S\ref{sec:eo4}, we run the Bockstein spectral sequences needed to compute the cohomology. In \S\ref{sec:AdamsDiff}, we run the Adams-Novikov spectral sequences for $eo_4$, $eo_4\wedge V(0)$, and $eo_4\wedge V(1)$. Finally, in \S\ref{sec:EOp-1Intro}, we describe more completely the geometric model that ties together the algebraic topology with the algebraic geometry initially considered by Hopkins-Gorbounov-Mahowald.

\section{Rational Computations and Main Results}\label{sec:HopfAlg}

\subsection{Rational Information}
The rational cohomology is easy to compute.

\begin{lem}
There are invariant classes $c_i$ of degree $2i(p-1)$ in $A$ such that
\[
H^*(A\otimes\Q,\Gamma\otimes\Q)=H^0(A\otimes\Q,\Gamma\otimes\Q)=\Q[c_2,\dots,c_p].
\]
\end{lem}

\begin{proof}
Since $p$ is a unit, we can transform Equation~\ref{eqn:DefEqn}
into one of the form
\[
y^{p-1}=x^p+a_2x^{p-2}+\dots+a_p
\]
by applying the morphism $x\mapsto x-\tfrac{a_1}{p}$. There are no translations in $x$ which preserve this form of the curve, so we conclude that $(A\otimes\Q,\Gamma\otimes\Q)$ is equivalent to the trivial sub-Hopf algebroid $(B,\Xi)$, where $B=\Xi=\mathbb Q[c_2',\dots,c_p']$ and $c_i'=\eta_R(a_i)$ evaluated at our choice of $r$. However, since all of the structure maps are the identity, $(B,\Xi)$ has no higher cohomology, and $H^0$ is just $B$. In particular, we conclude
\[
H^\ast(A,\Gamma)\cong H^\ast(B,\Xi)=B.
\]

It remains only to show that we can find classes $c_i$ in $A$ which also rationally generate the ring of invariants. The denominators of the elements $c_i'$ are powers of $p$, so we can multiply by a sufficiently high power of $p$ to get new generators that actually lie in $A$:
\[
c_i=\sum_{j=0}^{i}\binom{p-j}{i-j}(-1)^ja_{j}a_1^{i-j}p^{j}.\qedhere
\]
\end{proof}

\begin{cor}
As a ring,
\[
\pi_\ast(eo_{p-1})_\Q=\Q[c_2,\dots,c_p]
\]
\end{cor}

\subsection{Main Results}
The elements $c_i$ are rationally polynomial generators. In our $p$-local context, this means their products can be written as some power of $p$ times a sum of integral generators. To find the generators of $H^0(A,\Gamma)$, we have to add these and the obvious relations. The proof of the following theorem is our main goal.

\begin{thm}\label{Thm:MainThm}
As an algebra over $\Z_{(5)},$
\[
H^0(A,\Gamma)=\Z_{(5)}[c_2,c_3,\Delta_i,\Delta'_{15},\Delta'_{18},\Delta]/\big(\text{rels}\big),
\]
where $i$ ranges from $4$ to $22$, excluding $20$, where the degree of $\Delta_i$ and $\Delta'_i$ is $8i$, where $\Delta$ is the aforementioned discriminant, and where the expressions of these elements in terms of the elements $a_i$ and their relations are induced by the formulas from Table~\ref{Tab:GenRel}, together with the natural inclusion of $H^0(A,\Gamma)$ into $A$.
\end{thm}
Due to its length, Table~\ref{Tab:GenRel} is postponed until \S\ref{sec:relations}.

This ring has a distinguished ideal:
\[
\mathfrak m=(5,c_2,c_3,\Delta_i,\Delta_j').
\]
We shall see that this ideal also annihilates all of the higher cohomology.

\section{Preliminary, Prime Independent Remarks}\label{sec:AllPrimes}
\subsection{Reduction to a Simpler Hopf Algebroid}
The method of faithfully flat descent allows us to reduce to an equivalent but slightly smaller Hopf algebroid \cite{HoveyHopfAlgebroids}. The geometric underpinning is fairly simple. Up to a faithfully flat base change, any curve of the form
\[
y^{p-1}=x^p+a_1x^{p-1}+\dots+a_{p-1}x+a_p
\]
can be brought into the form
\[
y^{p-1}=x^p+a_1 x^{p-1}+\dots+a_{p-1}x
\]
by simply adding a root of the right-hand side and translating that root to zero. This condition places a strong restriction on the possible coordinate changes allowed, as now we can only translate by roots of the right hand side. In other words, up to faithfully flat base change, the curve and coordinate changes are corepresented by
\[
(\bar{A},\bar{\Gamma})=(\Z_p[a_1,\dots,a_{p-1}],\bar{A}[r]/r^p+a_1r^{p-1}+\dots+a_{p-1}r).
\]
We remark that $\bar{\Gamma}$ is the quotient of $\Gamma$ by the ideal generated by $a_p$ and $\eta_R(a_p)$. Since $(A,\Gamma)$ and $(\bar{A},\bar{\Gamma})$ flat-locally corepresent the same object, the underlying stacks are equivalent. This in turn implies that the Hopf algebroids have equivalent cohomologies. This is formally identical to the argument presented by Bauer for the cohomology of the Weierstrass Hopf algebroid at the prime $3$ \cite{Tilmantmf}.

\subsection{Invariant Ideals}

We will compute the Adams-Novikov spectral sequence via a sequence of Bockstein spectral sequences. It is clear from the formulation of the right units that the chain of ideals
\[
I_0=(p)\subset I_1=(p,a_1)\subset\dots\subset I_{p-1}=(p,a_1\dots,a_{p-1})
\]
is invariant. The quotients $(A/I_k,\Gamma/I_k)$ are therefore Hopf algebroids, and we can compute using a Miller-Novikov style algebraic Bockstein spectral sequence \cite{GreenBook}. If we filter by powers of these invariant ideals, we get spectral sequences of the form
\[
H^*(A/I_k,\Gamma/I_k)\otimes\Z_{(p)}[a_{k-1}]\Rightarrow H^*(A/I_{k-1},\Gamma/I_{k-1}).
\]
This is a trigraded spectral sequence of algebras. If the degree of a homogeneous element $x$ is written $(s,t,u)$, where $s$ is the cohomological degree, $t$ is the internal dimension, and $u$ is the Bockstein degree, then the degree
of $d_r(x)$ is $(s+1,t,u+r)$.

The first two Bockstein spectral sequences are the same for all primes greater than $2$.

\subsection{Computation of $H^*(A/I_{p-1},\Gamma/I_{p-1})$}
The Hopf algebroid $(A/I_{p-1},\Gamma/I_{p-1})$ is the Hopf algebra
\[
(\F_p,\F_p[r]/r^p).
\]
The cohomology of this is $\F_p[b]\otimes E(a)$, where $|a|=\big(1,2(p-1)\big)$, $|b|=\big(2,2p(p-1)\big)$, and in the cohomology of the bar complex,
\[
a=[r],\quad b=\toda{\underbrace{a,\dots,a}_p}=\sum_{i=1}^{p-1}\frac{1}{p}\binom{p}{j}[r^{p-i}|r^i].
\]
Since $r$ is primitive in $\Gamma$, we know that $a$ and $b$ represent classes in $H^*(A,\Gamma)$ and hence survive all of the algebraic spectral sequences. Topologically, $a$ represents the class $\alpha_1=h_{1,0}$, while $b$ represents $\beta_1=b_{1,0}$.

\subsection{Computation of $H^*(A/I_{p-2},\Gamma/I_{p-2})$}

We run the Bockstein spectral sequence for adding in $a_{p-1}$. The $E_1$-term is an exterior algebra on an element $a$ tensored with a polynomial algebra on elements $b$ and $a_{p-1}$. The tridegrees of these elements are
\[
|a|=\big(1,2(p-1),0\big),\,|b|=\big(2,2p(p-1),0\big),\, |a_{p-1}|=\big(0,2(p-1)^2,1\big).
\]
For dimension reasons, all of these are permanent cycles, so the spectral sequence collapses. Since the $E_\infty$-term is the free graded commutative algebra on $a$, $b$, and $a_{p-1}$, there are no hidden extensions.

\subsection{Computation of $H^*(A/I_{p-3},\Gamma/I_{p-3})$}

The $E_1$-term of this Bockstein spectral sequence is a polynomial algebra on the elements from the previous part, together with $a_{p-2}$. The tridegrees of the elements $a$ and $b$ are not changed, while the rest are:
\[
|a_{p-1}|=\big(0,2(p-1)^2,0\big),\, |a_{p-2}|=\big(0,2(p-1)(p-2),1\big).
\]
The formula for the right unit modulo $I_{p-3}$ shows that $a_{p-2}$ is a permanent cycle. It also shows that
\[
d_1(a_{p-1})=2aa_{p-2}.
\]

This leaves us the following algebra for the $E_2$-page:
\[
\left(\F_p[b,a_{p-1}^p,a_{p-2}]\otimes E(a)/aa_{p-2}\right)\{1,x_{1},\dots,x_{(p-1)}\}/(ax_k,a_{p-2}x_k),
\]
where $x_{k}$ has tridegree $(1,2(1+k(p-1))(p-1),0)$ and is represented by $aa_{p-1}^{k}.$

\begin{prop}
With the exception of $x_{p-1}$, all of the $x_{k}$ are non-bounding permanent cycles. We have $d_{p-1}(x_{p-1})\doteq a_{p-2}^{p-1}b.$
\end{prop}

\begin{proof}
For dimension reasons, the classes $x_k$ for $k<p-1$ are non-bounding permanent cycles.

The differential on $x_{p-1}$ follows from a Toda style argument \cite{GreenBook}:
\[
d_{p-1}(x_{p-1})\doteq\toda{a,\underbrace{d_1(a_{p-1}),\dots,d_1(a_{p-1})}_{p-1}}\doteq\toda{\underbrace{a,\dots,a}_{p}}a_{p-2}^{p-1}=ba_{p-2}^{p-1}.\qedhere
\]
\end{proof}

This gives the following $E_p$-page, which, for degree reasons, is also the $E_\infty$-page:
\[
\left(\F_p[b,a_{p-2},a_{p-1}^p]\otimes E(a)/(aa_{p-2},a_{p-2}^{p-1}b)\right)\{1,x_1,\dots,x_{p-2}\}/(ax_k,a_{p-2}x_k).
\]
There are also the following Massey product relations:
\[
\toda{x_{k},a,a_{p-2}}=x_{k+1}=\toda{a_{p-2}^{k+1},\underbrace{a,\dots,a}_{k+2}}.
\]
These in turn give multiplicative extensions between the elements $x_i$:
\[
x_ix_j=\begin{cases}
          a_{p-2}^{p-2}b & i+j=p-2 \\
      0 & \text{otherwise,}
         \end{cases}
\]
where $x_0=a$, by standard shuffling results \cite{MayMatricMassey}.

Since a curve of the form $y^{p-1}=x^p-a_{p-1}x$ is non-singular if and only if $a_{p-1}$ is a unit, we see that modulo $I_{p-3}$, the discriminant $\Delta$ is represented by $a_{p-1}^p$. We also already know this is an invariant class, so we know that $\Delta$ survives all of the Bockstein spectral sequences.

The class $x_1$ is also represented in the bar complex by $-r^p$. In particular, as long as we are working modulo $p$, this class is a cycle. This means it survives all of the algebraic spectral sequences except the last. Since $x_1$ bounds $pb$, topologically it represents the class $h_{1,1}$, a null-homotopy of $p\beta_1$.

We represent this $E_\infty$-page for the prime $5$ as Figure~\ref{fig:a3Einfty}. In this and in all following figures, the horizontal axis represents $t/8$ and the vertical axis represents $s$. This picture is repeated polynomially in $\Delta$, represented by a box, and also polynomially in $b$, so we will only list the first part.

\begin{figure}[ht]
\includegraphics{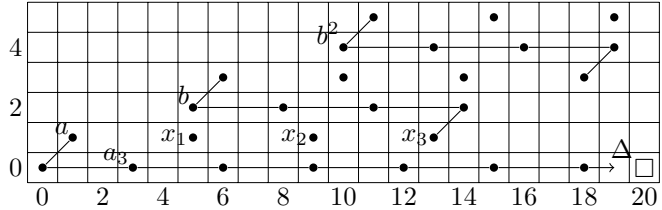}
\caption{$H^*(A/I_{p-3},\Gamma/I_{p-3})$} \label{fig:a3Einfty}
\end{figure}

In the picture, a solid line of positive slope is multiplication by $a$, and one of slope zero is multiplication by $a_3$. The case of the general prime is similar, except that the horizontal axis would be indexed as $t/2(p-1)$, and each row above the zeroth would have $(p-1)$ solid dots.

From this point on, we will restrict our attention to the prime $5$. In this case, we can find explicit representatives of the elements $x_{1},$ $x_{2}$, and $x_{3}$:
\[
x_1=a_4r+a_3r^2,\, x_{2}=a_4^2r+2a_3a_4r^2+3a_3^2r^3,\, x_{3}=a_4^3r+3a_3a_4^2r^2-a_3^2a_4r^3-3a_3^3r^4.
\]

\section{Computation of $H^*(A/I_1,\Gamma/I_1)$ and $V(1)_\ast eo_4$}
The element $a_1$ represents $v_1$, so the cohomology ring $H^\ast(A/I_1,\Gamma/I_1)$ is the Adams-Novikov $E_2$-page for the $V(1)$-homology of $eo_4$. Since $V(1)$ is a ring spectrum at the prime $5$, the Adams-Novikov spectral sequence for $V(1)_\ast eo_4$ will be a spectral sequence of algebras.

\subsection{Interlude: Notation}

For clarity, we will rely on pictures of the Bockstein $E_r$ terms to describe the initial situations and tell us which elements could support a differential.  In these Bockstein spectral sequences, the $d_r$-differential on any class must be divisible by the new, Bockstein element to the $r^{th}$ power. If we make the convention that a solid horizontal line means multiplication by the Bockstein element and an open circle means a polynomial algebra on this element, then we see that the possible targets of a $d_r$ differential are open circles preceded horizontally by $r$ solid lines. If we additionally make the convention that circles with dots in them are the non-Bockstein multiplicative generators, then the differentials are totally determined by their values on these elements. These conventions will allow us to immediately see which elements could support a differential.

\subsection{The $d_1$-differential}

We have a single differential coming immediately from the bar
complex:
\[
d_1(a_3)=3a_2a.
\]
Since $a_3a=0$, all elements of the form $a_3^k$ for $k>1$ are $d_1$-cycles.

\begin{figure}[ht]
\includegraphics{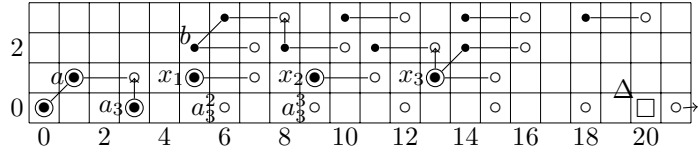}
\caption{$E_1$-page for $H^*(A/I_1,\Gamma/I_1)$} \label{fig:a2E1}
\end{figure}

There is one other $d_1$-differential:

\begin{prop}
We have $d_1(x_3)\doteq a_2a_3^2b.$
\end{prop}

\begin{proof}
The element $x_3$ can be written as
\[
x_3=\toda{a_3^3,a,a,a,a}.
\]
From this it follows from a simplification of May's work on Massey products, as presented in \cite{GreenBook}, that
\[
d_1(x_3)=\toda{d_1(a_3^3),a,a,a,a}=\toda{-a_2a_3^2a,a,a,a,a}=-a_2a_3^2b.\qedhere
\]\end{proof}

\subsubsection{The $d_2$-differential}
From the picture of the $E_2$-page (Figure \ref{fig:a2E2}), we see immediately that the only elements that can support a $d_2$-differential are $a_3^3$ and $x_2$.
\begin{figure}[ht]
\includegraphics{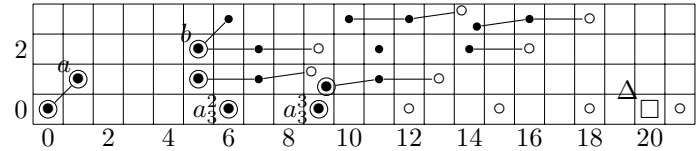}
\caption{The $E_2$-page for $H^*(A/I_1,\Gamma/I_1)$}
\label{fig:a2E2}
\end{figure}

\begin{prop}
$d_2(a_3^3)=-a_2^2x_1$ and $d_2(x_2)=-a_2^2b.$
\end{prop}

\begin{proof}
In the bar complex, we have
\[
a_3^3+3a_2a_3a_4\mapsto -a_2^2(a_2r^3+a_3r^2+a_4r)=-a_2^2x_1,
\]
giving the first differential, and
\[
a_4^2r+2a_3a_4r^2+3a_3^2r^3+2a_2a_4r^3+3a_2a_3r^4\mapsto -a_2^2b,
\]
giving the second.\end{proof}

\subsubsection{The $d_3$-differential and beyond}

Figure~\ref{fig:a2E3} shows a paucity of elements above the zero line. This implies that at this stage, the spectral sequence collapses.

\begin{figure}[ht]
\includegraphics{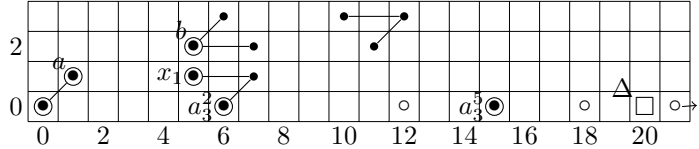}
\caption{$H^*(A/I_1,\Gamma/I_1)$}
\label{fig:a2E3}
\end{figure}

For computational reasons, we give here some of the full names for
some of the elements listed above. The elements $a$ and $b$ have
their usual bar representatives: $r$ and $\sum\tfrac{1}{p}\binom{p}{i} r^i | r^{p-i}$, while
\begin{align*}
x_1&=a_4r+a_3r^2+a_2r^3, \\
[a_3^2]&= a_3^2+2a_2a_4, \\
[a_3^5]&= a_3^5+2a_2^3a_3^3+a_2^4a_3a_4, \\
\Delta&=a_4^5-2a_3^4a_4^2-a_2a_3^2a_4^3+2a_2^2a_4^4+a_2^3a_3^2a_4^2+a_2^4a_4^3.
\end{align*}

With these elements, we can also compute the structure of $H^*$ as a
ring.

\begin{thm}\label{thm:V1eo4}
We have the multiplicative extension $2a[a_3^2]=a_2x_1,$ and the
full algebra of $H^*(A/I_1,\Gamma/I_1)$ is
\begin{multline*}
\F_5\big[a,b,x_1,a_2,[a_3^2],[a_3^5],\Delta] /\Big((a,x_1)^2,\,
a(a_2,[a_3^5]), \, a_2^2(b,x_1),\, b ([a_3^5],[a_3^2]^2),\\
([a_3^2]^5-[a_3^5]^2)-(a_2^3[a_3^2]^4+a_2^6[a_3^2]^3+2a_2^5\Delta),\,
2a[a_3^2]-a_2x_1\Big).
\end{multline*}
\end{thm}

\begin{proof}
The first part follows from noting that the difference of these two elements is the bar differential of $a_3a_4$. The algebra structure will follow from the first part by direct computation.
\end{proof}

We also record two Massey products which we will need for the Adams-Novikov spectral sequence.

\begin{lem}\label{lem:V1MasseyProducts}
We have Massey products
\[
x_1=\toda{a_2,a,a,a}
\]
and
\[
a_2b=\toda{x_1,a,a}.
\]
\end{lem}
\begin{proof}
The first is an easy, direct consequence of the previous Massey product form of $x_1$ involving $a_3$. The second bracket follows from the former by a Jacobi identity for Toda brackets proved by May \cite{MayMatricMassey}.
\end{proof}

Just as with $a$, $b$, and $x_1$, the class $[a_3^2]$ has a topological interpretation: this is $v_2$.

\section{Computing $H^*(A/I_0,\Gamma/I_0)$ and $V(0)_\ast eo_4$}\label{sec:V0eo4}
Because the spectral sequence is so sparse, this is actually easier to compute than the previous term.

\subsection{The $d_1$-differential}
We first note the differential coming immediately from the bar complex:
\[
d_1(a_2)= -a_1a.
\]
To continue, we use the picture of $E_1$ (Figure \ref{fig:a1E1}), marking this differential. We will use similar notation as before, but here solid lines with represent $a_1$ multiplications while dashed lines will represent $a_2$
multiplication. To further simplify the picture, we use a circled star to indicate a polynomial algebra on both $a_1$ and $a_2$.

\begin{figure}[ht]
\includegraphics{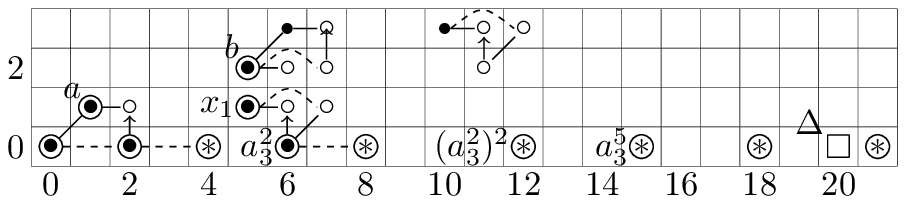}
\caption{$E_1$-page for $H^*(A/I_0,\Gamma/I_0)$} \label{fig:a1E1}
\end{figure}

\begin{prop}
We have $d_1([a_3^2])=3a_1x_1$.
\end{prop}

\begin{proof}
This is a fast computation in the bar complex. We saw that
$[a_3^2]$ is represented in the bar complex by $a_3^2+2a_2a_4.$ We
also have
\[
a_3^2\mapsto6a_2a_3r+2a_1a_3r^2+a_1a_2r^3+4a_2^2r^2+a_1^2r^4,
\]
while
\[
2a_2a_4\mapsto -a_2a_3r+a_2^2r^2+2a_1a_2r^3+3a_1a_4r+a_1a_3r^2+2a_1^2r^4.
\]
Adding these gives the result.
\end{proof}

This differential is suggested by to us by the topological story. Since $[a_3^2]$ represents $v_2$, in the story for $V(0)$, this class should kill $v_1 h_{1,1}$. We have already identified $v_1$ with $a_1$ and $h_{1,1}$ with $x_1$, predicting this differential.

\subsubsection{The $d_2$-differential and the $E_{\infty}$-page}

At this point, our spectral sequence is again very sparse (Figure
\ref{fig:a1E2}).

\begin{figure}[ht]
\includegraphics{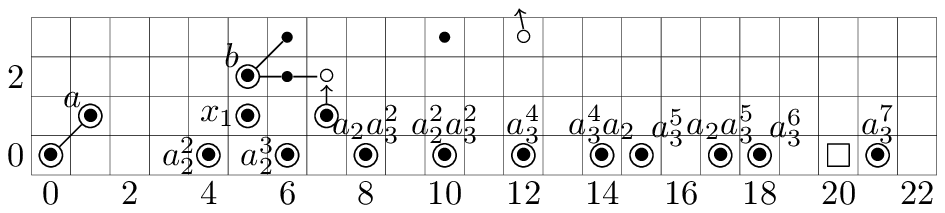}
\caption{$E_2$-Page of $H^*(A/I_0,\Gamma/I_0)$}
\label{fig:a1E2}
\end{figure}

We again see that we can have but a single differential.

\begin{prop}
We have $d_2(a_2x_1)=a_1^2b.$
\end{prop}

\begin{proof}
We start by computing the
bar differential on $a_2x_1=a_2^2r^3+a_2a_3r^2+a_2a_4r:$
\begin{multline*}
a_2x_1\mapsto
(2a_2a_3r|r)+(4a_1a_4r|r)+(3a_2^2r^2|r)+(3a_1a_3r^2|r)+(a_1a_2r^3|r)+(a_1^2r^4|r)
\\ +
(3a_2^2r|r^2)+(4a_1a_3r|r^2)+(3a_1a_2r^2|r^2)-(a_1^2r^3|r^2)-(2a_2a_3r|r) \\
+ (3a_1a_2r|r^3)+(a_1^2r^2|r^3)-(3a_2^2r^2|r)-(3a_2^2r|r^2).
\end{multline*}
If we add to $a_2x_1$ the element $-a_1a_2r^4+a_1a_3r^3+2a_1a_4r^2$, then the bar differential is exactly
$a_1^2b.$ \end{proof}

\begin{remark}
This differential is also predictable from Massey product considerations. Since the bracket defining $x_1$ contains a copy of $a_2$, we should expect a higher differential on $a_2x_1$ of the form
\[
a_2x_1\mapsto \toda{a_1a, \toda{a_1a,a,a,a}}.
\]
Shuffling and May's Jacobi identity then produce the desired description.
\end{remark}

It is clear that no further differentials are possible, so the spectral sequence collapses here. The result at this stage is a computation of the Adams-Novikov $E_2$-term for $V(0)_\ast eo_4$. We remark that just as for the $V(1)$ homology case, there is a hidden $a$ multiplication linking $x_1$ and $a_1b$. There is also a Massey product description of $x_1$:
\[
x_1=\toda{a_1,a,a,a,a}.
\]

\section{The $5$-local cohomology of $(A,\Gamma)$}\label{sec:eo4}
Everything we have done so far has led us to compute what happens
when we run the $5$-adic Bockstein spectral sequence. There is already an obvious
differential given by
\[
a_1\mapsto 5r,
\]
and since $5\neq 0$, we now have a differential on $x_1$:
\[
-x_1=r^5\mapsto 5r^4|r+10r^3|r^2+10r^2|r^3+5r|r^4=5b.
\]
This gives us all of the differentials for
dimension reasons, as we immediately see (Figure~\ref{fig:a0E1}).

\begin{figure}[ht]
\includegraphics{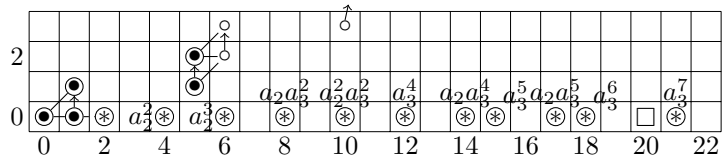}
\caption{$E_1$-Page for $H^*(A,\Gamma)$} \label{fig:a0E1}
\end{figure}

Since there are no more possible differentials, we conclude that $E_2=E_{\infty}.$ Additionally, we know the leading terms of the generators of $H^0$, since this is exactly what the Bockstein spectral sequences have been computing for us.

\begin{cor}
As an algebra, $H^0(A,\Gamma)$ is as described in Theorem
\ref{Thm:MainThm}.
\end{cor}

\begin{proof}
The Bockstein spectral sequences demonstrated that the classes given are the algebra generators. The relations are simple consequences of algebraic manipulations, so these are also immediate.
\end{proof}

Putting everything we have seen so far together allows us to show the following theorem.

\begin{thm}
As an algebra,
\[H^*(A,\Gamma)=H^0(A,\Gamma)[a,b]/\Big(a^2,\mathfrak{m}\cdot(a,b)\Big).\]
\end{thm}

\begin{proof}
The only surprise relation is $\mathfrak{m}\cdot(a,b)$, and this follows from the earlier fact that terms dominated in $(a_1,a_2,a_3)(a,b)$ were zero by the time we reached this last page.
\end{proof}

\section{Adams-Novikov Differentials}\label{sec:AdamsDiff}
All of the Adams-Novikov differentials arise from the Toda differential via standard techniques in stable homotopy theory. However, in the $V(1)$-homology case, there are some differentials of intermediate length arising from the Toda differential which the author had not previously seen.

\subsection{The $5$-local homotopy of $eo_4$}
In this section, we compute the Adams differentials for the homotopy of $eo_4$. This result follows easily from Hopkins and Miller's computation of the differentials for $EO_4$ \cite{HennFiniteKnResolutions}. The techniques and results are also very similar to those employed by Ravenel in \cite{RavenelArf}.

\begin{thm}
We have $d_9(\Delta)=ab^4.$
\end{thm}

\begin{proof}
The unit map $S^0\to EO_4$ factors through $eo_4$, so in particular $a$ corresponds to the homotopy element $\alpha_1\in \pi_7(S^0)$ and $b$ corresponds to $\beta_1\in\pi_{38}(S^0)$. The Toda relation $\alpha_1\beta_1^5$ in $\pi_\ast(S^0)$ guarantees that $ab^5$ must be zero in the homotopy of $eo_4$. For degree reasons, the only possible differential realizing this is
\[
d_9(b\Delta)=ab^5,
\]
giving the result.
\end{proof}

We can see hidden multiplicative extensions by considering the Massey product representatives of the ``left-over'' classes $[a\Delta]$, $[a\Delta^2]$, and $[a\Delta^3].$

\begin{prop}
The one-line of the $E_{10}$ page of Adams-Novikov spectral sequence is given by $a$, $[a\Delta^4]$, the classes
\begin{align*}
[a\Delta]&=\toda{\iota,ab^4,a} \\
[a\Delta^2]&=\toda{\iota,ab^4,ab^4,a} \\
[a\Delta^3]&=\toda{\iota,ab^4,ab^4,ab^4,a},
\end{align*}
and their $\Delta^5$ translates.
\end{prop}

\begin{prop}
We have hidden multiplicative extensions
\[
[a\Delta^{i}]\cdot[a\Delta^{j}]=\begin{cases}
b^{13} & \text{when }i+j=3, \\
0 & \text{otherwise.}
\end{cases}
\]
\end{prop}

\begin{proof}
The hidden extension follows by May's Jacobi identity and classical shuffling results.
\end{proof}

Our $d_9$-differential gives rise via a Toda style argument to a higher differential.

\begin{prop}
We have $d_{33}([a\Delta^4])\doteq b^{17}.$
\end{prop}

The spectral sequence collapses at this point, as there are not
enough things in higher filtration to be the target of any further
differentials.

\begin{cor}
As an algebra,
\[
\pi_\ast(EO_4)/\mathfrak m\cong \F_5[\Delta]\otimes P,
\]
where $P$ is a Poincar\'e duality algebra of dimension $|b^{16}|=608$.
\end{cor}

\begin{remark}
Hopkins and Miller have shown that the differentials in this section are basically the same for all $p$. As a consequence, the multiplicative extensions we derived and the decomposition of $\pi_\ast(EO_{p-1})/\mathfrak m$ into a polynomial algebra and a Poincar\'e duality algebra are similar for all primes. This is an incarnation of Gross-Hopkins and Mahowald-Rezk dualities \cite{GrossHopkinsDuality, MahowaldRezkFP}, similar to what is seen in the homotopy of $tmf$ at $p=2$ and $p=3$.
\end{remark}

The differentials for $eo_4$ imply the bulk of those for the $V(0)$ and $V(1)$ homologies.

\subsection{The $V(0)$-homology of $eo_4$}
Since the Adams-Novikov spectral sequence for the $V(0)$-homology is a module over the Adams-Novikov spectral sequence for the homotopy, we know that $ab^4$ and $b^{17}$ must be zero by the $E_\infty$ page. The source of the differential which truncates these $b$- towers must be $b$-torsion free. The only classes which are $b$-torsion free on the $E_2$ page are $1$, $a$, $x_1$, $a_1x_1$, and $\Delta$. For degree reasons, $a$, $x_1$, and $a_1x_1$ are permanent cycles. Thus the primary differential is the one on $\Delta$ from the homotopy.

\begin{lem}
We have a $d_9$-differential
\[
d_9(\Delta)=a b^4.
\]
\end{lem}

This yields two additional differentials.

\begin{cor}
We have two $d_{33}$-differentials
\[
d_{33}(a_1\Delta^4)=\toda{a_1,ab^4,ab^4,ab^4,ab^4}=x_1b^{16}
\]
and
\[
d_{33}(a\Delta^4)=\toda{a,ab^4,ab^4,ab^4,ab^4}=b^{17}.
\]
\end{cor}

\subsection{The $V(1)$-homology of $eo_4$}

A similar argument to the $V(0)$ case applies here. The class $ab^4$ and $b^{17}$ must be zero. For degree reasons, we again see that there is only one possible differential realizing this: the differential from the sphere.

\begin{lem}
We again have a $d_9$-differential
\[
d_9(\Delta)=ab^4.
\]
\end{lem}

The Massey products described in Lemma~\ref{lem:V1MasseyProducts} give slight generalizations of the Toda implied differentials.

\begin{prop}
We have a $d_{17}$-differential
\[
d_{17}(x_1\Delta^2)\doteq\toda{x_1,ab^4,ab^4}=a_2b^9.
\]
\end{prop}

\begin{prop}
We have a $d_{25}$-differential
\[
d_{25}(a_2\Delta^3)\doteq\toda{a_2,ab^4,ab^4,ab^4}=x_1b^{12}.
\]
\end{prop}

Finally, there is the usual Toda differential
\begin{cor}
There is a $d_{33}$-differential
\[
d_{33}(a\Delta^4)\doteq b^{17}.
\]
\end{cor}

Again, for degree reasons, the spectral sequence at this point collapses.

\section{Methods and Formulas Relating the classes $\Delta_i$}\label{sec:relations}
The formulas herein presented are by no means unique. In most cases, we can find a different collection of algebraic generators by modifying any of the given generators by decomposables. The method of finding these is somewhat involved but relatively straightforward. We begin with formulas for the classes $c_i$ at the prime $5$:
\begin{align*}
c_2&=-2a_1^2+5a_2,\\
c_3&=4a_1^3-15a_1a_2+25a_3,\\
c_4&=-3a_1^4+15a_1^2a_2-50a_1a_3+125a_4,\\
c_5&=4a_1^5-25a_1^3a_2+125a_1^2a_3-625a_1a_4+3125a_5.
\end{align*}

We attach a kind of $5$-adic depth to each generator of $H^0(A,\Gamma)$, describing the order of the additive subgroup of $H^0$ generated by that class modulo the ideal generated by $c_2$ through $c_5$. The leading terms coming from the Bockstein spectral sequences come in pairs of the form $a_1^\epsilon a_2^ia_3^ja_4^k$, where $\epsilon$ is $0$ or $1$. The case $\epsilon=0$ is slightly easier. Since $c_n\cong 5^{n-1} a_n$ modulo $a_1$ for $2\leq n\leq 4$, we see that the generator with leading term $a_2^ia_3^ja_4^k$ must be of the form
\[\frac{1}{5^{i+2j+3k}}c_2c_3c_4+\dots.\] The number $i+2j+3k$ is the desired depth. The case of $\epsilon=1$ is similar, and the depth is again $i+2j+3k$.

This notion of depth allows us to simplify our search for generators in a fixed degree. We begin by considering all products of generators of lower degree. Our notion of depth is appropriately additive, in that the depth of $ab$ is that of $a$ plus that of $b$, so we can assign depths to all of these products. We consider the collection of all such products with depth less than the depth of the generator for which we are searching. We then build our new class by finding relations among those classes with the largest depth and working our way down. This greatly speeds the search in practice.

\begin{table}[ht]
\begin{tabular}{ll}
$\Delta_4$ & $\tfrac{1}{25}\big(4c_4+3c_2^2\big)$\\

$\Delta_5$ & $\tfrac{1}{25}\big(2c_5+c_2c_3\big)$ \\

$\Delta_6$ & $\tfrac{1}{125}\big(2c_3^2-4c_2c_4+c_2^3\big)$ \\

$\Delta_7$ & $\tfrac{1}{5}\big(c_3\Delta_4-2c_2\Delta_5\big)$ \\

$\Delta_8$ & $\tfrac{1}{5^5}\big(-3c_2c_3^2+9c_2^2c_4-4c_4^2+3c_3c_5\big)$ \\

$\Delta_9$ &
    $\tfrac{1}{5^5}\big(-9c_3^3+32c_2c_3c_4-9c_2^2c_5+4c_4c_5\big)$ \\

$\Delta_{10}$ &  $\tfrac{1}{200}\big(2\Delta_5^2+c_2\Delta_4^2-15\Delta_4\Delta_6\big)$ \\

$\Delta_{11}$ &
$\tfrac{1}{10}\big(3\Delta_5\Delta_6-\Delta_4\Delta_7\big)$ \\

$\Delta_{12}$ &  $\tfrac{1}{5^8}\big(54c_3^4-279c_2c_3^2c_4+216c_2^2c_4^2-224c_4^3+81c_2^2c_3c_5+144c_3c_4c_5-27c_2c_5^2\big)$  \\

$\Delta_{13}$ & $\tfrac{1}{15}\big(\Delta_4\Delta_9-4\Delta_5\Delta_8\big)$ \\

$\Delta_{14}$ &  $\frac{1}{50}\big(4c_4\Delta_{10}-6c_3\Delta_{11}+30\Delta_6\Delta_8-15\Delta_4\Delta_{10}+15c_2\Delta_{12}\big)$ \\

$\Delta_{15}'$ & $\tfrac{1}{5}\big(\Delta_5\Delta_{10}-2\Delta_4\Delta_{11}\big)$ \\

$\Delta_{15}$ & $\tfrac{1}{25}\big(2\Delta_6\Delta_9-\Delta_7\Delta_8+5\Delta_{15}'-15c_2\Delta_{13}\big)$ \\

$\Delta_{16}$ &  $\tfrac{1}{25}\big(-4\Delta_5\Delta_{11}-c_2\Delta_4\Delta_{10}+15\Delta_6\Delta_{10}-15c_2\Delta_{14}\big)$ \\

$\Delta_{17}$ &  $\tfrac{1}{25}\big(-3c_3\Delta_{14}-2c_2\Delta_{15}'+20\Delta_8\Delta_9\big)$ \\

$\Delta_{18}$ &  $\tfrac{1}{5}\big(2\Delta_5\Delta_{13}-\Delta_4^2\Delta_{10}+2\Delta_4\Delta_6\Delta_8\big)$ \\

$\Delta_{18'}$ &  $\tfrac{1}{25}\big(2\Delta_9^2+16c_2\Delta_8^2-95\Delta_8\Delta_{10}\big)$ \\

$\Delta_{19}$ & $\tfrac{1}{5}\big(8\Delta_8\Delta_{11}-\Delta_9\Delta_{10}\big)$ \\

$\Delta_{21}$ & $\tfrac{1}{25}\big(2c_3\Delta_{18}'+12 c_2\Delta_{19}-30\Delta_{10}\Delta_{11}+15\Delta_9\Delta_{12}\big)$ \\

$\Delta_{22}$ & $\tfrac{1}{25}\big(c_3\Delta_{19}+c_4\Delta_{18}'+10\Delta_9\Delta_{13}-5\Delta_{11}^2\big)$ \\

$\Delta$ & \begin{tabular}{l}
    $\frac{1}{5^{15}}\big(-100c_2^3c_3^2c_4^2-135c_3^4c_4^2+400c_2^4c_4^3+720c_2c_3^2c_4^3$ \\
    $-640c_2^2c_4^4+256c_4^5+80c_2^3c_3^3c_5+108c_3^5c_5-360c_2^4c_3c_4c_5$ \\
    $-630c_2c_3^3c_4c_5+560c_2^2c_3c_4^2c_5-320c_3c_4^3c_5+108c_2^5c_5^2$ \\
    $+165c_2^2c_3^2c_5^2-180c_2^3c_4c_5^2+90c_3^2c_4c_5^2+80c_2c_4^2c_5^2$ \\
    $-30c_2c_3c_5^3+c_5^4\big)$
    \end{tabular}

\end{tabular}
\caption{Generators and Basic Relations for $H^0(A,\Gamma)$}
\label{Tab:GenRel}
\end{table}

\section{Relationship with $EO_{p-1}$}\label{sec:EOp-1Intro}
Manin showed that the Jacobian of the Artin-Schreier curve over $\F_p$
\[
y^{p-1}=x^p-x
\]
admits a one dimensional formal summand of height $n=(p-1)$ \cite{ManinFG}. Since this is the first height at which the Morava stabilizer group has elements of order $p$, Hopkins, Mahowald, and Gorbounov used this fact to build a geometric model analogous to the story of elliptic curves and $tmf$ at the prime $3$, and they show that the formal completion of the Jacobians of the family of curves given by Equation~(\ref{eqn:DefEqn}) over the Witt vectors of $\F_{p^n}$ has a summand which is the Lubin-Tate universal deformation of the Honda formal group \cite{GorMah}.

There are three moduli problems associated with height $p-1$, all of which are related to the curves described by Equation~(\ref{eqn:DefEqn}). We will first describe the full moduli stack, $\mathcal M_{p-1}$, as this is closely related to the computation. For reference and completeness, we will also describe an open substack corresponding to smooth curves and a formal neighborhood of a point in this curve. Historically, the analysis has been in the reverse order. The Hopkins-Gorbounov-Mahowald result concerns itself mainly with the formal neighborhood, and the extension over the entire moduli stack is ongoing work.

The main moduli stack in question represents curves which can (flat--)locally be described by Equation~(\ref{eqn:DefEqn}), together with their automorphisms which fix the point at infinity. This should be thought of as a generalization of the Deligne-Mumford compactification of the moduli stack of elliptic curves to heights greater than two. Since the stack is determined by local data, and since this is determined by the coefficients in Equation~(\ref{eqn:DefEqn}), this moduli stack is the stack associated to the (ungraded) Hopf algebroid
\[
(\hat{A},\hat{\Gamma})=(\Z_p[a_1,\dots,a_p],\hat{A}[r,\lambda^{\pm 1}]).
\]
The $\G_m$ action on the curves amounts to a grading on the Hopf algebroid, recovered by viewing Equation~(\ref{eqn:DefEqn}) as a homogeneous equation of degree $2p(p-1)$. When we pass to the graded version, we recover $(A,\Gamma)$.

This moduli stack has an open substack $\mathcal M_{p-1}[\Delta^{-1}]$ defined by adding the requirement that the curves be smooth. As was noted earlier, a curve defined by Equation~(\ref{eqn:DefEqn}) is smooth if and only if $\Delta$ is invertible. This means that this moduli stack is the stack associated to the Hopf algebroid $(A[\Delta^{-1}], \Gamma[\Delta^{-1}])$.

We can associate to each curve defined by Equation~(\ref{eqn:DefEqn}) its Jacobian, and the scaling action splits off a summand of the $p$-divisible group. This family of $p$-divisible groups is the essential data needed to apply Lurie's derived algebraic geometry version of Artin representability \cite{BehrensLawsonTAF} and thereby to produce a sheaf of $E_\infty$ ring spectra over this moduli stack with global sections defined to be $eo_{p-1}[\Delta^{-1}]$. The Adams-Novikov $E_2$-term for the homotopy of $eo_{p-1}[\Delta^{-1}]$ is given by the cohomology of the Hopf algebroid $(A[\Delta^{-1}],\Gamma[\Delta^{-1}])$. Since $\Delta$ is an invariant element of $A$ and since $\Delta$ acts faithfully on the cohomology of $(A,\Gamma)$, this cohomology is easy to compute from our results.

\begin{thm}
As an algebra,
\[
\Ext_{(A[\Delta^{-1}],\Gamma[\Delta^{-1}])}(A[\Delta^{-1}],A[\Delta^{-1}])=\Ext_{(A,\Gamma)}(A,A)[\Delta^{-1}],
\]
and for $0\leq n\leq 3$, we also have
\[
\Ext_{(A[\Delta^{-1}]/I_n, \Gamma[\Delta^{-1}]/I_n)}(A[\Delta^{-1}]/I_n, A[\Delta^{-1}]/I_n)= \Ext_{(A/I_n, \Gamma/I_n)}(A/I_n,A/I_n)[\Delta^{-1}].
\]
\end{thm}
\begin{proof}
In all of the above spectral sequences, multiplication by $\Delta$ is faithful on every page. We conclude that for all of the spectral sequences the map from $\Ext$ to the localized $\Ext$ is flat, and the result follows.
\end{proof}

We conclude moreover that the map on homotopy induced by $eo_{p-1}\to eo_{p-1}[\Delta^{-1}]$ is simply inverting $\Delta$ and is an injection. Thus all of the surprising multiplicative relations we saw in the homotopy of $eo_{p-1}$ persist to $eo_{p-1}[\Delta^{-1}]$.

Finally inside this open substack is a point corresponding to the curve $y^{p-1}=x^p-x$ considered by Manin. The formal neighborhood of this point in the stack is the stack defined by the Hopf algebroid
\[
(A[\Delta^{-1}]_{I}^{\wedge},\Gamma[\Delta^{-1}]_{I}^{\wedge}),
\]
where $I$ is the kernel of the map $A\to \F_p$ defining the curve in question. This is the moduli problem originally considered by Hopkins-Gorbounov-Mahowald. The Jacobian of the universal curve over the formal neighborhood has a one-dimensional formal summand isomorphic to the universal deformation of the Honda formal group law, and the Hopkins-Miller theorem then produces a sheaf of $E_\infty$-ring spectra over this formal moduli stack, the global sections of which are the Hopkins-Miller real $K$-theory spectrum $EO_{p-1}$. Just as with the passage to the smooth substack, our knowledge of the cohomology of $(A,\Gamma)$ produces for us an understanding of the cohomology of the localized, completed Hopf algebroid.

\begin{thm}
As algebras,
\[
\Ext_{(A[\Delta^{-1}]_{I}^{\wedge},\Gamma[\Delta^{-1}]_{I}^{\wedge})}(A[\Delta^{-1}]_{I}^{\wedge},A[\Delta^{-1}]_{I}^{\wedge})=\Ext_{(A,\Gamma)}(A,A)[\Delta^{-1}]_{\mathfrak m}^{\wedge}.
\]
\end{thm}
\begin{proof}
An argument essentially identical to the Behrens-Lawson argument identifying the $K(n)$-localization of $TAF$ shows that the spectrum $eo_{p-1}[\Delta^{-1}]$ is $E(n)$-local and that $L_{K(n)}eo_{p-1}[\Delta^{-1}]=EO_{p-1}$. A result of Hovey and Strickland shows that this $K(n)$-localization is completion with respect to $p$ and $v_i$ for $i<n$ \cite{HoveyStricklandK(n)}.

For $EO_{4}$, Hopkins, Gorbounov, and Mahowald identify this completion in the Adams-Novikov $E_2$-term with $\Ext$ over the $I$-adic completion of $(A[\Delta^{-1}],\Gamma[\Delta^{-1}])$. This also allows us to see that in the Adams-Novikov $E_2$-term for $eo_4[\Delta^{-1}]$, completion with respect to $5$, $v_1$, $v_2$, and $v_3$ coincides with $\mathfrak m$-adic completion, which gives the desired result.
\end{proof}

\bibliography{biblio}

\end{document}